%% file: eddyCurrent.tex
\documentclass[12pt,reqno]{amsart}

\usepackage{amsmath,amssymb,ifthen}
\usepackage{fullpage}
\usepackage{graphicx,psfrag,subfigure}
\usepackage{color}
\usepackage{comment}

\usepackage{moreverb}

\input{mymacros.tex}

\begin{document}

\title[Eddy-Current-LLG]%
{On a decoupled linear FEM integrator for Eddy-current-LLG}
\date{\today}

\author{Kim-Ngan~Le}
\author{Marcus~Page}
\author{Dirk~Praetorius}
\author{Thanh~Tran}
\address{Institute for Analysis and Scientific Computing,
       Vienna University of Technology,
       Wiedner Hauptstra\ss{}e 8-10,
       A-1040 Wien, Austria}
\email{dirk.praetorius@tuwien.ac.at}
\email{marcus.page@tuwien.ac.at}
\address{School of Mathematics and Statistics,
         The University of New South Wales,
         Sydney 2052, Australia}
\email{n.le-kim@student.unsw.edu.au}
\email{thanh.tran@unsw.edu.au}


\begin{abstract}
We propose a numerical integrator for the coupled system of
the eddy-current equation with the nonlinear
Landau--Lifshitz--Gilbert equation. The considered effective
field contains a general field contribution, and we particularly cover exchange, 
anisotropy, applied field, and magnetic field (stemming from the eddy-current equation). Even though the considered problem is nonlinear, our scheme requires only the solution of two linear systems per time-step. Moreover, our algorithm decouples both equations so that in each time-step, one linear system is solved for the magnetization, and afterwards one linear system is solved for the magnetic field. Unconditional convergence -- at least of a subsequence -- towards a weak solution
is proved, and our analysis even provides existence of such weak solutions. Numerical experiments with a micromagnetic
benchmark problem underline the performance of the
proposed algorithm.
\end{abstract}


\maketitle
\input{1_introduction}
\input{2_problem}
\input{3_prelim}
\input{4_algorithm}
\input{5_convergence}
\input{6_numerics}

\noindent \textbf{Acknowledgements.}
The authors acknowledge financial support through the
Austrian projects WWTF MA09-029, FWF P21732, and
the Australian ARC project DP120101886.

\def\new#1{#1}

\input{literature.tex}


\end{document}

%% file: mymacros.tex
\def\H{\widetilde{H}}

\def\R{{\mathbb R}}

\def\EE{{\mathcal E}}

\def\KK{{\mathcal K}}

\def\MM{{\mathcal M}}

\def\PP{{\mathcal P}}

\def\SS{{\mathcal S}}
\def\SSS{{\mathbb S}}
\def\TT{{\mathcal T}}
\def\NN{\mathcal{N}}

\def\XX{{\mathcal X}}

\def\II{{\mathcal I}}
\def\L{\mathbf{L}}
\def\H{\mathbf{H}}

\newcommand{\norm}[3][]{#1\|#2#1\|_{#3}}

\def\set#1#2{\big\{#1\,:\,#2\big\}}

\def\eps{\varepsilon}

\newcounter{const}
\def\setc#1{\refstepcounter{const}\label{const:#1}C_{\theconst}}
\def\c#1{C_{\ref{const:#1}}}

\newcounter{contractionnumber}
\def\nameq#1#2{%
  \ifthenelse{\equal{#1}{reduction}}{q_{\rm red}}{%
  \ifthenelse{\equal{#1}{estconv}}{q_{\rm est}}{%
  \ifthenelse{\equal{#1}{cea}}{q_{\mbox{\scriptsize C\'ea}}}{%
  \ifthenelse{\equal{#2}{newcounter}}{\refstepcounter{contractionnumber}\label{contraction#1}}{}q_{\ref{contraction#1}}}%
}}}

\def\Ce{C_e}

\def\mmm{\mathbf{m}}
\def\EEE{\mathbf{E}}
\def\HHH{\mathbf{H}}

\def\heff{\mathbf{H}_{\text{eff}}}

\def\subsub{\Subset}
\def\nnn{\mathbf{n}}
\def\ppp{\mathbf{p}}
\def\zzz{\mathbf{z}}
\def\xxx{\mathbf{x}}
\def\vvv{\mathbf{v}}

\def\pphi{\pmb{\phi}}

\def\ppsi{\pmb{\psi}}

\def\vphi{\pmb{\varphi}}
\def\zzeta{\pmb{\zeta}}

\def\dt{d_t}

\def\curl{\text{\textbf{curl}}}
\def\wgamma{\gamma_{hk}}
\DeclareMathOperator{\diver}{div}

\newtheorem{theorem}{Theorem}

\newtheorem{lemma}[theorem]{Lemma}

\newtheorem{algorithm}[theorem]{Algorithm}
\newtheorem{definition}[theorem]{Definition}
\newtheorem{remark}[theorem]{Remark}

\def\subsection#1
{
 \bigskip

 \refstepcounter{subsection}
 {\noindent\bf\arabic{section}.\arabic{subsection}.~#1.~}
}

%% file: 1_introduction.tex
\section{Introduction}\label{sec:intro}

The Landau--Lifshitz--Gilbert equation (LLG) has been
widely used to model micromagnetic phenomena which have
applications in the production of magnetic sensors,
recording heads, and magneto-resistive storage devices
\cite{Gil55, LL35}. Existence and non-uniqueness results can be found 
in \cite{as, visintin}. In our contribution, the LLG equation is
coupled with the quasi-static Maxwell's equations
to describe
electromagnetic wave and magnetization propagation of a
ferromagnetic medium confined in a larger magnetic
field.

Throughout the literature, various works on the numerical
analysis of LLG and coupling to the full Maxwell system can be found,
and we refer to~\cite{alouges2008,alouges2011,banas,maxwell,bjp,bp} 
and the references therein. 
Considering the quasi-static approximation of 
the Maxwell system, also known as the eddy-current equation (E), however, 
only little work has been done. 

In \cite{lt}, the analysis of~\cite{alouges2008} is successfully
extended 
to the study of the coupled eddy-current and Landau-Lifshitz-Gilbert system (ELLG), for a simplified effective field.
There, a convergent linear integrator was developed which, however, needs
the solution of one huge linear system for the coupled problem. On the other hand,
in~\cite{maxwell}, an algorithm for the Maxwell-LLG system is presented which 
decouples both problems and requires the solution of two small linear systems per 
time-step. In the present paper, we combine the ideas of~\cite{maxwell} and~\cite{lt}
to derive an unconditionally convergent algorithm for the ELLG system which
decouples both problems. The proposed algorithm requires the successive solution
of only two small linear systems, one for LLG- and one for the eddy-current part. This improvement has a huge impact on the computational applicability of the scheme since an existing LLG solver can easily be reused. This simplifies implementation as well as possible debugging. Moreover, possible
preconditioning of the eddy-current part greatly benefits from the decoupling as well. Finally,
we introduce a general field operator $\pi(\cdot)$ which allows us to cover much more general field contributions than previous works. In particular, our work covers exchange,
anisotropy, and external field contributions, as well as the magnetic field from the eddy-current part. We emphasize that, with the techniques from~\cite{multiscale}, a spatial approximation of the effective field can rigorously be included into the analysis.

The remainder of the paper is organized as follows. In
Section~\ref{sec:problem} we give the precise problem formulation as well as the notion of a weak solution. Section~\ref{sec:prelim} is devoted to the
introduction of finite element spaces and their
approximation properties. The algorithm is presented in
Section~\ref{sec:algo}, and the main result on convergence
is presented and proved in
Section~\ref{sec:convergence}. Finally, 
Section~\ref{sec:numerics} is devoted to our numerical
results.

%% file: 2_problem.tex
\section{Problem formulation}\label{sec:problem}
We consider the Landau-Lifshitz-Gilbert equation coupled with the eddy-current equation. This system describes the evolution of the magnetization of a ferromagnetic body that occupies the domain $\omega \subsub \Omega \subseteq \R^3$. For a given damping parameter $\alpha > 0$, the magnetization $\mmm :(0,T) \times \omega \rightarrow \SSS^2$ and the magnetic field $\HHH:(0,T) \times \Omega \rightarrow \R^3$ satisfy the ELLG system
\begin{subequations}\label{eq:mllg}
\begin{align}
&\mmm_t - \alpha \mmm \times \mmm_t = - \mmm \times \heff \quad \text{ in } \omega_T := (0,T) \times \omega\label{eq:mllg1}\\
&\mu_0 \HHH_t + \sigma^{-1} \nabla \times (\nabla \times \HHH) = -\mu_0 \mmm_t \quad \text{ in } \Omega_T := (0,T) \times \Omega\label{eq:mllg2}
\end{align}
where the effective field $\heff$ consists of $\heff = C_e \Delta \mmm + \HHH + \pi(\mmm)$ for some general time-independent energy contribution $\pi : \L^2(\Omega) \to \L^2(\Omega)$, which is assumed to fulfill a certain set of properties, see~\eqref{eq:assum2}--\eqref{eq:assum3}.  We stress that, with the techniques from~\cite{multiscale}, an approximation $\pi_h$ of $\pi$ can rigorously be included into the analysis as well, see Section~\ref{sec:prelim} below. 
Furthermore, we emphasize that throughout this work, the case $\heff = C_e \Delta \mmm + \HHH + C_a D\Phi(\mmm) + \HHH_{ext}$ is particularly covered. Here, $\Phi(\cdot)$ denotes the crystalline anisotropy density, and $\HHH_{ext}$ is a given applied field. The constant $\mu_0 \ge 0$ denotes the magnetic permeability of free space, and the constant $\sigma \ge 0$ stands for the conductivity of the ferromagnetic domain $\omega$. As is usually done for simplicity, we assume $\Omega \subset \R^3$ to be bounded with perfectly conducting outer surface $\partial \Omega$ into which the ferromagnet $\omega \subsub \Omega$ is embedded, and $\Omega \backslash \overline \omega$ is assumed to be vacuum. Additionally, the ELLG system~\eqref{eq:mllg} is supplemented by initial conditions
\begin{align}\label{eq:init}
\mmm(0,\cdot) = \mmm^0  \text{ in } \omega \quad \text{ and } \quad  \HHH(0,\cdot) = \HHH^0  \text{ in } \Omega
\end{align}
as well as boundary conditions
\begin{align}\label{eq:boundary}
\partial_\nnn\mmm = 0  \text{ on } \partial \omega_T, \qquad (\nabla \times \HHH)\times\nnn = 0  \text{ on } \partial \Omega_T.
\end{align}
The space $\HHH(\curl; \Omega)$ is defined in Section~\ref{sec:prelim}. Note that the side constraint $|\mmm| = 1$ a.e.\ in $\omega_T$  
directly follows from $|\mmm^0| = 1$ a.e.\ in $\omega$ and $\partial_t |\mmm|^2 = 2 \mmm \cdot \mmm_t = 0$ in $\omega_T$, which is a consequence of~\eqref{eq:mllg1}. This behaviour should also be reflected by the numerical integrator.
In analogy to~\cite{lt},
we assume the given data to satisfy
\begin{align}\label{eq:data}
\mmm^0 \in H^1(\omega, \SSS^2), \qquad \HHH^0 \in  \HHH(\curl; \Omega)
\end{align}
as well as
\begin{align}\label{eq:consistency}
\diver(\HHH^0 + \chi_\omega \mmm^0) = 0 \quad \text{ in } \Omega, \qquad \langle\HHH^0 + \chi_\omega\mmm^0, \nnn\rangle = 0 \quad \text{ on } \partial \Omega.
\end{align}
\end{subequations}
We now recall the notion of a weak solution of~\eqref{eq:mllg1}--\eqref{eq:mllg2} from~\cite{lt} which extends~\cite{as} from the pure LLG to ELLG .
\begin{definition}\label{def:weak_sol}
Given~\eqref{eq:data}--\eqref{eq:consistency}, the tupel $(\mmm, \HHH)$ is called a weak solution of ELLG if,
\begin{itemize}
\item[(i)] $\mmm \in \H^1(\omega_T)$ with $|\mmm| = 1$ almost everywhere in $\omega_T$;
\item[(ii)] $\HHH, \HHH_t, \nabla \times \HHH \in \L^2(\Omega_T)$, i.e.\ $\HHH \in H^1(\L^2) := H^1([0,T]; \L^2(\Omega))$ and $\nabla \times \HHH \in \L^2(\Omega_T)$ in the weak sense;
\item[(iii)] for all $\vphi \in C^\infty(\omega_T)$ and $\zzeta \in C^\infty(\Omega_T)$, we have
\begin{align}\label{eq:weak_sol1}
\int_{\omega_T} \langle\mmm_t, \vphi\rangle - \alpha \int_{\omega_T} \langle(\mmm \times \mmm_t), \vphi\rangle &= -C_e\int_{\omega_T} \langle(\nabla \mmm\times\mmm), \nabla\vphi\rangle\\\nonumber
&\quad+ \int_{\omega_T}\langle(\HHH \times \mmm),\vphi\rangle + \int_{\omega_T}\langle(\pi(\mmm)\times \mmm), \vphi\rangle, \\
\mu_0\int_{\Omega_T}\langle\HHH_t, \zzeta \rangle + \sigma^{-1}\int_{\Omega_T}\langle\nabla \times \HHH, \nabla \times \zzeta \rangle &= - \mu_0\int_{\omega_T} \langle\mmm_t, \zzeta\rangle;\label{eq:weak_sol2}
\end{align}
\item[(iv)] there holds $\mmm(0,\cdot) = \mmm^0$ and $\HHH(0, \cdot) = \HHH^0$ in the sense of traces;
\item[(v)] for almost all $t' \in [0,T]$, we have bounded energy
\begin{align}\label{eq:energy}
\norm{\nabla \mmm(t')}{\L^2(\omega)}^2 + \norm{\mmm_t}{\L^2(\omega_{t'})}^2 &+ \norm{\HHH(t')}{\L^2(\Omega)}^2 + \norm{(\nabla \times \HHH)(t')}{\L^2(\Omega)}^2 + \norm{\HHH_t}{\L^2(\Omega_{t'})}^2
 \le \c{energy},
\end{align}
where $\setc{energy}>0$ is independent of $t'$.
\end{itemize}
\end{definition}

Existence of weak solutions for a simplified effective field was first shown in~\cite{lt}. Moreover, existence also follows from the current work as our analysis is constructive.

\begin{remark}
In the special case $\heff = \Delta \mmm + \HHH$, the energy estimate~\eqref{eq:energy} becomes
\begin{align*}
\EE(t') + \norm{\mmm_t}{\L^2(\Omega_{t'})}^2 + \norm{\HHH_t}{\L^2(\Omega_{t'})}^2 + \norm{\nabla \times \HHH}{\L^2(\Omega_{t'})}^2 \le \EE(0),
\end{align*}
with 
\begin{align*}
\EE(t') = \norm{\nabla \mmm (t')}{\L^2(\omega)}^2 + \norm{\HHH(t')}{\L^2(\Omega))}^2 + \norm{(\nabla \times \HHH)(t')}{\L^2(\Omega)}^2.
\end{align*}
Moreover, under some additional assumptions on the general operator $\pi(\cdot)$, namely boundedness in $\L^4(\Omega)$ and self-adjointness, one can even derive
\begin{align*}
\EE(t') + C \norm{\mmm_t}{\L^2(\Omega_{t'})}^2 + \norm{\HHH_t}{\L^2(\Omega_{t'})}^2 + \norm{\nabla \times \HHH}{\L^2(\Omega_{t'})}^2 \le \EE(0)
\end{align*}
for the full effective field
see~\cite{diss}.
\end{remark}

\begin{remark}
We emphasize the additional regularity $\HHH_t \in \L^2(\Omega_T)$ and $\nabla \times \HHH \in \L^2(\Omega_T)$ for the derivative and the curl of the magnetic field $\HHH$. If LLG is coupled to the full Maxwell system, the current analysis of weak solvers provides only the reduced regularity $\EEE, \HHH \in \L^2(\Omega_T)$ for the electric and magnetic field, see~\cite{banas, maxwell}.
\end{remark}

%% file: 3_prelim.tex
\section{Preliminaries}\label{sec:prelim}
For time discretization, we impose a uniform partition $0 = t_0 < t_1 < \hdots < t_N = T$ of the time interval $[0,T]$. The time-step size is denoted by $k=k_j := t_{j+1} - t_j$ for $j = 0, \hdots, N-1$. For each (discrete) function $\vphi$, we denote by $\vphi^j := \vphi(t_j)$ the evaluation at time $t_j$. 
Furthermore, we write $d_t \vphi^{j+1} := (\vphi^{j+1} - \vphi^j)/k$ for $j \ge 1$ and a sequence $\{\vphi^j\}_{j\ge 0}$.

For the spatial discretization, let $\TT_h^{\Omega}$ be a regular triangulation of the polyhedral bounded Lipschitz domain $\Omega \subset \R^3$ into compact and non-degenerate tetrahedra. By $\TT_h$, we denote its restriction to $\omega \subsub \Omega$, where we assume that $\omega$ is resolved, i.e.\
\begin{align*}
\TT_h = \TT_h^{\Omega}|_\omega = \set{T \in \TT_h^{\Omega}}{T \cap \omega \neq \emptyset} \quad \text{and} \quad \overline \omega = \bigcup_{T \in \TT_h} T.
\end{align*}
By $\SS^1(\TT_h)$, we denote the standard $\PP^1$-FEM space of globally continuous and piecewise affine functions from $\omega$ to $\R^3$, i.e.
\begin{align*}
\SS^1(\TT_h) := \{\pphi_h \in C(\overline \omega, \R^3) : \pphi_h|_T \in \PP_1(T) \text{ for all } T \in \TT_h\}.
\end{align*}
By $\II_h: C(\Omega) \to \SS^1(\TT_h)$, we denote the nodal interpolation operator onto this space.
The set of nodes of the triangulation $\TT_h$ is denoted by $\NN_h$.
To discretize the magnetization $\mmm$ in~\eqref{eq:mllg1}, we define the set of admissible discrete magnetizations by
\begin{align*}
\MM_h := \{\pphi_h \in \SS^1(\TT_h) : |\pphi_h(\zzz)| = 1 \text{ for all } \zzz \in \NN_h\}.
\end{align*}
The main idea in the upcoming algorithm is to introduce an additional free variable $\vvv$ for the time derivative of $\mmm$, since LLG is a linear equationin $\vvv = \mmm_t$. 
Due to the modulus constraint $|\mmm(t)| = 1$, and therefore $\mmm_t\cdot \mmm = 0$ almost everywhere in $\omega_T$, we discretize the time derivative $\vvv(t_j) := \mmm_t(t_j)$ in the discrete tangent space which is defined by
\begin{align*}
\KK_{\pphi_h} := \{\ppsi_h \in \SS^1(\TT_h) : \ppsi_h(\zzz) \cdot \pphi_h(\zzz) = 0 \text{ for all } \zzz \in \NN_h\}
\end{align*}
for any $\pphi_h \in \MM_h$. For two vectors $\mathbf{x}, \mathbf{y} \in \R^3, \mathbf{x}\cdot \mathbf{y}$ stands for the usual scalar product in $\R^3$.

To discretize the eddy-current equation~\eqref{eq:mllg2}, we follow the lines of~\cite{lt} and use the conforming ansatz spaces $\XX_h \subset \HHH(\curl;\Omega) := \set{\vphi \in \L^2(\Omega)}{\nabla \times \vphi \in \L^2(\Omega)}$, given by the first order edge elements, i.e.\
\begin{align*}
\XX_h := \{\vphi_h \in \HHH(\curl; \Omega) : \vphi_h|_T \in \PP_1(T) \text{ for all } T \in \TT_h^{\Omega}\},
\end{align*}
cf.~\cite[Chapter 8.5]{monk}. Associated with $\XX_h,$ let $\II_{\XX_h}: \H^2(\Omega) \to \XX_h$ denote the corresponding nodal FEM interpolator.
By standard estimates, see e.g. \cite{monk, brennerscott}, one derives the approximation property
\begin{align}\label{eq:interp}
&\norm{\vphi - \II_{\XX_h}\vphi}{\L^2(\Omega)} + h\norm{\nabla \times (\vphi - \II_{\XX_h}\vphi)}{\L^2(\Omega)} \le C\, h^2\norm{\nabla ^2 \vphi}{\L^2(\Omega)}
\end{align}
for all $\vphi \in \H^2(\Omega)$. Here and throughout, $h > 0$ denotes the maximal element diameter of the elements $T \in \TT_h$.

As for the general field contribution, we assume that $\pi$ is a spatial operator which maps the magnetization $\mmm(t) \in \L^2(\Omega)$ at given time $t$ onto some field $\pi(\mmm)(t) = \pi(\mmm(t))\in \L^2(\Omega)$, i.e.~$\pi(\cdot)$ is not time-dependent. As mentioned above, it is even possible to replace $\pi$ by some numerical approximation $\pi_h$ as long as a certain weak convergence property is fulfilled, cf.~\cite[Equation (32)]{multiscale}. In particular, this includes approximation errors, arising from numerical computation of complicated field contributions, into the analysis. 

Finally, given two expressions $A$ and $B$, we write $A \lesssim B$ if there exists a constant $c>0$ which is independent of $h$ and $k$, such that $A \le cB$. 

%% file: 4_algorithm.tex
\section{Numerical algorithm}\label{sec:algo}

We recall that the  LLG equation~\eqref{eq:mllg1} can equivalently be stated as
\begin{align}\label{eq:llg:algo}
\alpha \mmm_t + \mmm \times \mmm_t = \heff-(\mmm \cdot \heff)\mmm
\end{align}
under the constraint $|\mmm| = 1$ almost everywhere in $\Omega_T$.
This formulation will now be used to construct the upcoming numerical scheme, where we follow the approaches of~\cite{alouges2008, alouges2011, multiscale, petra, mathmod2012, gamm2011}. Note that in contrast to~\cite{lt}, our integrator fully decouples LLG from the eddy-current equation which greatly simplifies an actual numerical implementation as well as the possible preconditioning of iterative solvers.

\begin{algorithm}\label{alg}
\begin{itemize}
Input: Initial data $\mmm^0$ and $\HHH^0$, parameter $0 \le \theta \le 1$, counter $i = 0$. For all $i=0, \hdots, N-1$ iterate:
\item[(i)] Compute unique solution $\vvv_h^i \in \KK_{\mmm_h^i}$ such that for all $\pphi_h \in \KK_{\mmm_h^i}$ there  holds
\begin{subequations}
\begin{equation}\label{eq:alg:1}
\begin{split}
&\alpha(\vvv_h^i, \pphi_h) + \big((\mmm_h^i \times \vvv_h^i), \pphi_h\big) + C_e(\theta k\nabla \vvv_h^i, \nabla \pphi_h) = -C_e (\nabla\mmm_h^i, \nabla \pphi_h) \\&\hspace*{35ex}+ (\HHH_h^{i}, \pphi_h) + \big(\pi(\mmm_h^i), \pphi_h\big).
\end{split}
\end{equation}
\item[(ii)] Define $\mmm_h^{i+1} \in \MM_h$ nodewise by $\displaystyle \mmm_h^{i+1}(\zzz) = \frac{\mmm_h^i(\zzz) + k\vvv_h^i(\zzz)}{|\mmm_h^i(\zzz) + k\vvv_h^i(\zzz)|}$ for all $\zzz \in \NN_h$.
\item[(iii)] Compute unique solution $\HHH_h^{i+1} \in \XX_h$ such that for all $\zzeta_h \in \XX_h$ there holds
\begin{align}\label{eq:alg:2}
&\hspace*{4ex}\mu_0(d_t\HHH_h^{i+1}, \zzeta_h) + \sigma^{-1} (\nabla \times \HHH_h^{i+1}, \nabla \times \zzeta_h) = -\mu_0(\vvv_h^i, \zzeta_h).
\end{align}
\end{subequations}
\end{itemize}
\end{algorithm}

The following lemma states that the above algorithm is indeed well-defined. 
\begin{lemma}\label{lem:solve:alg}
Algorithm~\ref{alg} is well-defined in the sense that it admits a unique solution $(\vvv_h^i, \mmm_h^{i+1}, \HHH_h^{i+1})$ at each step $i = 0, \hdots, N-1$ of the iterative loop. Moreover, we have $\norm{\mmm_h^i}{\L^\infty(\omega)} = 1$ for each $i = 0, \hdots, N$.
\end{lemma}

\begin{proof}
Unique solvability of~\eqref{eq:alg:1}--\eqref{eq:alg:2} directly follows from the linearity of the right-hand sides, positive definiteness of the left-hand sides, and finite space dimension, cf.\ e.g.~\cite{maxwell}. Due to the Pythagoras theorem and the pointwise orthogonality from $\KK_{\mmm_h^i}$, we further get $|\mmm_h^i(\zzz) + k \vvv_h^i(\zzz)|^2 = |\mmm_h^i(\zzz)|^2 + k|\vvv_h^i(\zzz)|^2 \ge 1$, and thus also step $(ii)$ of the algorithm is well-defined. The boundedness of $\norm{\mmm_h^i}{\L^\infty(\omega)}=1$ finally follows from normalization at the grid points and use of barycentric coordinates.
\end{proof}

\begin{remark}
At first glance, it might seem a bit odd that the notion of a weak solution and the construction of the numerical scheme rely on different formulations of LLG. Besides the fact that the weak solution was already formulated in earlier works, one would expect that the algorithm even converges to a tupel $(\mmm, \HHH)$ that fulfills a formulation of a weak solution based on equation~\eqref{eq:llg:algo}. Suprisingly, however, this is not the case as an additional term occurs. For details, the reader is referred to~\cite{diss}.
\end{remark}

%% file: 5_convergence.tex
\section{Main theorem \& Convergence analysis}\label{sec:convergence}
In this section, we consider the convergence properties of the above algorithm and show that it indeed converges towards a weak solution of the coupled ELLG system. Moreover, the proof is constructive in the sense that it even shows existence of weak solutions of ELLG.

\subsection{Main result}\label{sec:mainresult}
We start by collecting some general assumptions. Throughout, we assume that the spatial meshes $\TT_h$ are uniformly shape regular and satisfy the angle condition
\begin{align}\label{eq:assum1}
\int_\omega \nabla \zeta_i \cdot \nabla \zeta_j \le 0 \quad \text{ for all hat functions } \zeta_i, \zeta_j \in \SS^1(\TT_h) \text{ with } i \neq j.
\end{align}
For $\xxx \in \Omega$ and $t \in [t_i, t_{i+1})$, we now define for $\gamma_h^\ell \in \{\mmm_h^\ell, \HHH_h^\ell, \vvv_h^\ell\}$ the time approximations
\begin{equation}\label{eq:discrete_functions}
\begin{split}
&\wgamma(t, \xxx) := \frac{t-t_i}{k}\gamma_h^{i+1}(\xxx) + \frac{t_{i+1} - t}{k}\gamma_h^i(\xxx),\\
&\wgamma^-(t, \xxx):= \gamma_h^{i}(\xxx),\quad \wgamma^+(t,\xxx):= \gamma_h^{i+1}(\xxx),
\end{split}
\end{equation}
and note the $\partial_t \gamma_{hk}(t,\xxx) = \dt \gamma_h^{i+1}(\xxx)$.

\begin{remark}\label{rem:en_decay}
The angle condition~\eqref{eq:assum1} is automatically fulfilled for tetrahedral meshes with dihedral angle smaller than $\pi/2$. It is needed to ensure the discrete energy decay
$
\int_\omega \big|\nabla\II_h\big(\frac{\mmm_h}{|\mmm_h|}\big)\big|^2 \le \int_\omega |\nabla \mmm_h|^2$, for the nodal interpoland $\II_h:C(\overline \Omega) \to \SS^1(\TT_h)$ and all $\mmm_h \in \SS^1(\TT_h)$ with $|\mmm_h(\zzz)| \ge 1$ for all $\zzz \in \NN_h$,
cf.~\cite{bartels}.
\end{remark}

The next statement is the main result of this work.

\begin{theorem}\label{thm:convergence}

\textbf{(a)} Suppose that there exists a constant $C_\pi > 0$ which only depends on $|\omega|$ such that the general energy contribution $\pi(\cdot)$ is uniformly bounded
\begin{align}\label{eq:assum2}
\norm{\pi(\nnn)}{\L^2(\omega)}^2 \le C_\pi, \quad \text{ for all } \nnn \in \L^2(\omega) \text{ with } \norm{\nnn}{\L^2(\omega)}^2 \le 1.
\end{align}
Moreover, for the initial data, we assume
\begin{align}\label{eq:assum4}
\mmm_h^0 \rightharpoonup \mmm^0 \quad \text{ weakly in } \H^1(\omega), \quad \text{ as well as } \quad
\HHH_h^0 \rightharpoonup \HHH^0 \text{ weakly in } \HHH(\emph{\curl}, \Omega).
\end{align} 
Then, we have strong subconvergence of $\mmm_{hk}^-$ towards some function $\mmm$ in $\L^2(\Omega_T)$.

\bigskip

\noindent \textbf{(b)} In addition to the above, we assume
\begin{align}\label{eq:assum3}
\pi(\mmm_{hk}^-) \rightharpoonup \pi(\mmm) \quad \text { weakly subconvergent in $\L^2(\omega_T)$}.
\end{align}
Then, the computed FE solutions $(\mmm_{hk}, \HHH_{hk})$ are weakly subconvergent in $\H^1(\omega_T) \times \big(H^1(\L^2(\Omega)) \cap L^2(\HHH(\emph{\curl}, \Omega))\big)$ towards a weak solution $(\mmm, \HHH)$ of ELLG. In particular, this yields existence of weak solutions and each accumulation point of $(\mmm_{hk}, \HHH_{hk})$ is a weak solution in the sense of Definition~\ref{def:weak_sol}.

\end{theorem}

\begin{remark}
The conditions~\eqref{eq:assum2} and~\eqref{eq:assum3} are fulfilled for all field contributions mentioned in Section~\ref{sec:problem}. Moreover, those conditions are fulfilled by the operators arising from certain (nonlinear) multiscale problems, as well as their respective numerical discretizations, cf.~\cite{multiscale}.
\end{remark}

The proof of the main Theorem~\ref{thm:convergence} will roughly be done in three steps:
\begin{enumerate}
\item[(i)] Boundedness of the discrete quantities and energies.
\item[(ii)] Existence of weakly convergent subsequences.
\item[(iii)] Identification of the limits with a weak solution of ELLG.
\end{enumerate}

\begin{lemma}\label{lem:stability}
For all $k < \alpha$, the discrete quantities $(\mmm_h^j, \HHH_h^{j}) \in \MM_h \times \XX_h$ fulfill 
\begin{align}\label{eq:discrete_energy}\nonumber
\norm{\nabla \mmm_h^j}{\L^2(\omega)}^2 +& k\sum_{i=0}^{j-1}\norm{\vvv_h^i}{\L^2(\omega)}^2  + \big(\theta - 1/2\big)k^2 \sum_{i=0}^{j-1}\norm{\nabla \vvv_h^i}{\L^2(\omega)}^2 + \norm{\HHH_h^j}{\L^2(\Omega)}^2 + \norm{\nabla \times \HHH_h^j}{\L^2(\Omega)}^2 \\\nonumber
& + \sum_{i=0}^{j-1}\norm{\HHH_h^{i+1} - \HHH_h^i}{\L^2(\Omega)}^2 + k \sum_{i=0}^{j-1} \norm{\dt \HHH_h^{i+1}}{\L^2(\Omega)}^2 + k \sum_{i=0}^{j-1}\norm{\nabla \times \HHH_h^{i+1}}{\L^2(\Omega)}^2 \\
&+ \sum_{i=0}^{j-1}\norm{\nabla \times (\HHH_h^{i+1} - \HHH_h^i)}{\L^2(\Omega)}^2
 \le \c{en_dis}
\end{align}
for each $j = 0, \hdots, N$ and some constant $\setc{en_dis}>0$ that only depends on $|\Omega|$, on $|\omega|$, as well as on $C_\pi$.
\end{lemma}
\begin{proof}
For the eddy-current equation~\eqref{eq:alg:2} in step (iii) of Algorithm~\ref{alg}, we choose $\zzeta_h= \HHH_h^{i+1}$ as test function and multiply by $\frac{k}{\Ce}$ to get
\begin{align}\label{eq:tmp}
&\frac{\mu_0}{\Ce}(\HHH_h^{i+1} - \HHH_h^i, \HHH_h^{i+1}) + \frac{k }{\sigma\Ce} \norm{\nabla \times \HHH_h^{i+1}}{\L^2(\Omega)}^2 = -\frac{\mu_0 k}{\Ce}(\vvv_h^i, \HHH_h^{i}) + \frac{\mu_0k}{\Ce} (\vvv_h^i, \HHH_h^{i} - \HHH_h^{i+1}).
\end{align}
The LLG equation~\eqref{eq:alg:1} is tested with $\vphi_h = \vvv_h^i \in \KK_{\mmm_h^i}$. With $\big((\mmm_h^i \times \vvv_h^i), \vvv_h^i\big) = 0$, this yields after multiplication with $\frac{\mu_0k}{C_e} >0$
\begin{align*}
\frac{\mu_0\alpha k}{C_e} \norm{\vvv_h^i}{\L^2(\omega)}^2 + \mu_0\theta k^2 \norm{\nabla \vvv_h^i}{\L^2(\omega)}^2 = - \mu_0 k (\nabla \mmm_h^i, \nabla \vvv_h^i) + \frac{\mu_0 k}{C_e} (\HHH_h^i, \vvv_h^i) + \frac{\mu_0k}{C_e}\big(\pi(\mmm_h^i), \vvv_h^i\big).
\end{align*}
Next, we follow the lines of~\cite{alouges2008} and use the fact that $\norm{\nabla \mmm_h^{i+1}}{\L^2(\omega)}^2 \le \norm{\nabla(\mmm_h^i + k\vvv_h^i)}{\L^2(\omega)}^2$, cf.\ Remark~\ref{rem:en_decay}, to see
\begin{equation}\label{eq:nabla_m_bounded}
\begin{split}
\frac{\mu_0}{2}\norm{\nabla \mmm_h^{i+1}}{\L^2(\omega)}^2 
&\le \frac{\mu_0}{2} \norm{\nabla \mmm_h^i}{\L^2(\omega)}^2 + \mu_0k\,(\nabla \mmm_h^i, \nabla \vvv_h^i)+ \frac{\mu_0 k^2}{2}\norm{\nabla \vvv_h^i}{\L^2(\omega)}^2 \\
&\le \frac{\mu_0}{2} \norm{\nabla \mmm_h^i}{\L^2(\omega)}^2 - \mu_0\big(\theta - 1/2\big)k^2\norm{\nabla \vvv_h^i}{\L^2(\omega)}^2\\
&\quad - \frac{\alpha\mu_0 k}{C_e}\norm{\vvv_h^i}{\L^2(\omega)}^2 + \frac{\mu_0k}{C_e}(\HHH_h^i, \vvv_h^i) + \frac{\mu_0k}{C_e}\big(\pi(\mmm_h^i), \vvv_h^i\big).
\end{split}
\end{equation}
Combining~\eqref{eq:tmp}--\eqref{eq:nabla_m_bounded}, we obtain
\begin{align*}
\frac{\mu_0}{2}(\norm{\nabla \mmm_h^{i+1}}{\L^2(\omega)}^2 - \norm{\nabla \mmm_h^i}{\L^2(\omega)}^2) &+ \mu_0(\theta -1/2)k^2\norm{\nabla \vvv_h^i}{\L^2(\omega)}^2 + \frac{\alpha \mu_0 k}{\Ce}\norm{\vvv_h^i}{\L^2(\omega)}^2 \\
& \quad+ \frac{\mu_0}{\Ce}(\HHH_h^{i+1} - \HHH_h^i, \HHH_h^{i+1}) + \frac{k}{\sigma\Ce}\norm{\nabla \times \HHH_h^{i+1}}{\L^2(\Omega)}^2 \\
& \le \frac{\mu_0 k}{\Ce}(\vvv_h^i, \HHH_h^i - \HHH_h^{i+1}) + \frac{\mu_0 k}{\Ce}(\pi(\mmm_h^i), \vvv_h^i).
\end{align*}
Next, we recall Abel's summation by parts, i.e.\ for arbitrary $u_i \in \R$ and $j \ge 0$, there holds
\begin{align}\label{eq:abel_sum}
\sum_{i=1}^j(u_i - u_{i-1},u_i) = \frac12 |u_j|^2 - \frac12|u_0|^2 + \frac12\sum_{i=1}^j|u_i - u_{i-1}|^2.
\end{align}
Summing up over $i=0, \hdots, j-1$, and exploiting Abel's summation for the $\HHH_h^i$ scalar product as well as the inequalities of Young and H\"older, this yields for any $\eps > 0$
\begin{align*}
&\frac{\mu_0}{2} \norm{\nabla \mmm_h^j}{\L^2(\omega)}^2 +\big(\theta - 1/2\big)\mu_0k^2\sum_{i=0}^{j-1}\norm{\nabla \vvv_h^i}{\L^2(\omega)}^2+ \frac{\alpha k\mu_0}{C_e}\sum_{i=0}^{j-1}\norm{\vvv_h^i}{\L^2(\omega)}^2 + \frac{\mu_0}{2\Ce}\norm{\HHH_h^j}{\L^2(\Omega)}^2\\
&\quad + \frac{\mu_0}{2\Ce}\sum_{i=0}^{j-1}\norm{\HHH_h^{i+1} - \HHH_h^i}{\L^2(\Omega)}^2 + \frac{k}{\sigma\Ce}\sum_{i=0}^{j-1}\norm{\nabla \times \HHH_h^{i+1}}{\L^2(\Omega)}^2\\
&\le \frac{\mu_0 k}{4
\eps\Ce}\sum_{i=0}^{j-1}(\norm{\pi(\mmm_h^i)}{\L^2(\omega)}^2
+ \norm{\HHH_h^{i+1} - \HHH_h^i}{\L^2(\Omega)}^2) +
\frac{\eps\mu_0
k}{\Ce}\sum_{i=0}^{j-1}\norm{\vvv_h^i}{\L^2(\omega)}^2
\\
& \quad + \frac{\mu_0}{2}(\norm{\nabla \mmm_h^0}{\L^2(\omega)} + \frac{1}{\Ce}\norm{\HHH_h^0}{\L^2(\Omega)}^2).
\end{align*}
With the notation
$
C_\vvv^k := \frac{2\mu_0 k}{C_e}(\alpha - \eps), \text{ and } C_\HHH^k := \frac{\mu_0}{C_e}\big(1 - \frac{k}{2\eps}\big),
$
this yields
\begin{equation}\label{proof:tmp2} 
\begin{split}
&\mu_0 \norm{\nabla \mmm_h^j}{\L^2(\omega)}^2 +2\big(\theta - 1/2\big)\mu_0k^2\sum_{i=0}^{j-1}\norm{\nabla \vvv_h^i}{\L^2(\omega)}^2 + C_\vvv^k\sum_{i=0}^{j-1}\norm{\vvv_h^i}{\L^2(\omega)}\\
&\quad + \frac{\mu_0}{C_e}\norm{\HHH_h^{j}}{\L^2(\Omega)}^2 + C_\HHH^k\sum_{i=0}^{j-1}\norm{\HHH_h^{i+1} - \HHH_h^i}{\L^2(\Omega)}^2+ \frac{2k}{\sigma\Ce}\sum_{i=0}^{j-1}\norm{\nabla \times \HHH_h^{i+1}}{\L^2(\Omega)}^2\\
&\le \frac{\mu_0 k}{2 \eps\Ce}\sum_{i=0}^{j-1}\norm{\pi(\mmm_h^i)}{\L^2(\omega)}^2 + \mu_0\norm{\nabla \mmm_h^0}{\L^2(\omega)} + \frac{\mu_0}{\Ce}\norm{\HHH_h^0}{\L^2(\Omega)}^2.
\end{split}
\end{equation}
Next, we test with $\zzeta_h = \dt \HHH_h^{i+1}$ in~\eqref{eq:alg:2} to obtain after multiplication by $2k$
\begin{align*}
2\mu_0 k \norm{\dt \HHH_h^{i+1}}{\L^2(\Omega)}^2 + 2 \sigma^{-1} (\nabla \times \HHH_h^{i+1}, \nabla \times (\HHH_h^{i+1} - \HHH_h^i)) = -2\mu_0 k (\vvv_h^i, \dt \HHH_h^{i+1}).
\end{align*}
The right-hand side can further be estimated by
\begin{align*}
-2\mu_0 k(\vvv_h^i, \dt \HHH_h^{i+1}) \le \mu_0 k \norm{\vvv_h^i}{\L^2(\omega)}^2 + \mu_0 k \norm{\dt \HHH_h^{i+1}}{\L^2(\Omega)}^2.
\end{align*}
Abel's summation by parts~\eqref{eq:abel_sum} thus yields
\begin{equation}\label{proof:tmp3}
\begin{split}
\mu_0 k \sum_{i=0}^{j-1}\norm{\dt \HHH_h^{i+1}}{\L^2(\Omega)}^2 + \sigma^{-1} \norm{\nabla \times \HHH_h^j}{\L^2(\Omega)}^2 &+ \sigma^{-1} \sum_{i=0}^{j-1}\norm{\nabla \times (\HHH_h^{i+1} - \HHH_h^i)}{\L^2(\Omega)}^2\\
& \le \sigma^{-1}\norm{\nabla \times \HHH_h^0}{\L^2(\Omega)}^2 + \mu_0 k \sum_{i=0}^{j-1}\norm{\vvv_h^i}{\L^2(\omega)}^2.
\end{split}
\end{equation}
Finally, we weight~\eqref{proof:tmp3} by $\alpha/C_e$ and add~\eqref{proof:tmp2}. The last term on the right-hand side of~\eqref{proof:tmp3} can be absorbed by the corresponding term on the left-hand side of~\eqref{proof:tmp2}. For the desired result, we have to ensure that there is a choices of $\eps$ such that the $C_\vvv^k - \mu_0 k \alpha/C_e$, and $C_\HHH^k$ are positive, i.e.\ $
(\alpha - 2\eps)> 0$ and $\big(1 - \frac{k}{2\eps}\big) > 0$. This is, however,
equivalent to $k/2 < \eps < \alpha/2$.
From the assumed convergence of the initial data~\eqref{eq:assum4} as well as~\eqref{eq:assum2}, we know that the right-hand side is uniformly bounded, which concludes the proof.
\end{proof}

We can now conclude the existence of weakly convergent subsequences.

\begin{lemma}\label{lem:subsequences}
There exist functions $(\mmm, \HHH) \in \H^1(\omega_T)\times \big(H^1(\L^2) \cap L^2(\HHH(\emph{\curl}))\big)$, with $|\mmm| = 1$ almost everywhere in $\omega$ such that up to extraction of a subsequence, there holds
\begin{subequations}\label{eq:subsequences}
\begin{align}
&\mmm_{hk} \rightharpoonup \mmm \text{ in } \H^1(\omega_T),\label{eq:subsequences0}\\
&\mmm_{hk}, \mmm_{hk}^\pm \rightharpoonup \mmm \text{ in } L^2(\H^1(\omega)),  \label{eq:subsequences1}\\
&\mmm_{hk}, \mmm_{hk}^\pm \rightarrow \mmm \text{ in } \L^2(\omega_T),\label{eq:subsequences2}\\
&\HHH_{hk} \rightharpoonup \HHH \text{ in } H^1(\L^2(\Omega)) \cap L^2(\HHH(\emph{\curl}, \Omega)),\\
&\HHH_{hk}^\pm \rightharpoonup \HHH \text{ in } L^2(\HHH(\emph{\curl}, \Omega)),\label{eq:subsequences3}\\
&\vvv_{hk}^- \rightharpoonup \mmm_t \text{ in } \L^2(\omega_T). \label{eq:subsequences4}
\end{align}
\end{subequations}
Here, the subsequences are constructed successively, i.e.\
for arbitrary mesh-sizes $h \rightarrow 0$, and time-step sizes $k \rightarrow 0$ there exist subindices $h_\ell, k_\ell$ for which the above convergence properties~\eqref{eq:subsequences} are satisfied simultaneously.
\end{lemma}
\begin{proof}
Analogously to~\cite[Lemma 9]{maxwell} and~\cite[Lemma 4.4]{lt}, the proof of~\eqref{eq:subsequences0}--\eqref{eq:subsequences3} directly follows from the boundedness of the discrete quantities from Lemma~\ref{lem:stability} in combination with the continuous inclusions $\H^1(\omega_T) \subseteq L^2(\H^1(\omega)) \subseteq \L^2(\omega_T)$ and $H^1(\L^2(\Omega)) \cap L^2(\HHH(\curl, \Omega)) \subseteq \L^2(\Omega_T)$. For~\eqref{eq:subsequences0}, we additionally exploited the inequality $\norm{\mmm_h^{i+1} - \mmm_h^i}{\L^2(\Omega)}^2 \le k^2\norm{\vvv_h^i}{\L^2(\Omega)}^2$, cf.~\cite{alouges2008}. From $\norm{\partial_t \mmm_{hk}(t) - \vvv_{hk}^-(t)}{\L^2(\Omega)} \lesssim k \norm{\vvv_{hk}^-(t)}{\L^2(\Omega)}^2$ (see~\cite{alouges2008}) and lower semi-continuity, we deduce~\eqref{eq:subsequences4}. The normalization of the limiting function $\mmm$ finally follows by direct calculation, i.e.\
\begin{align*}
\norm{|\mmm| - 1}{\L^2(\omega_T)} \le \norm{|\mmm| - |\mmm_{hk}^-|}{\L^2(\omega_T)} + \norm{|\mmm_{hk}^-| - 1}{\L^2(\omega_T)}
\end{align*}
and 
\begin{align*}
\norm{|\mmm_{hk}^-(t,\cdot)| - 1}{\L^2(\omega)} \le h \max_{t_j} \norm{\nabla \mmm_h^j}{\L^2(\omega)}.
\end{align*}
This concludes the proof.
\end{proof}

Now, we have collected all ingredients for the proof of our main theorem.

\begin{proof}[Proof of Theorem~\ref{thm:convergence}]
Let $\vphi \in C^\infty(\omega_T)$ and $\zzeta \in C^\infty(\Omega_T)$ be arbitrary. We now define test functions by $(\pphi_h, \zzeta_h)(t, \cdot) := \big(\II_h(\mmm_{hk}^-\times \vphi), \II_{\XX_h}\zzeta\big)(t,\cdot)$. Obviously, for any $t \in [t_j, t_{j+1})$, we have $(\pphi_h, \zzeta_h) \in (\KK_{\mmm_h^j}, \XX_h)$.
With the notation~\eqref{eq:discrete_functions}, Equation~\eqref{eq:alg:1} of Algorithm~\ref{alg} implies
\begin{align*}
\alpha \int_0^T (\vvv_{hk}^-, \pphi_h) + \int_0^T\big((\mmm_{hk}^- \times \vvv_{hk}^-), \pphi_h\big) &= -C_e\int_0^T\big(\nabla(\mmm_{hk}^- + \theta k \vvv_{hk}^-), \nabla \pphi_h)\big)\\
&\qquad \qquad  + \int_0^T (\HHH_{hk}^-, \pphi_h) + \int_0^T\big(\pi(\mmm_{hk}^-), \pphi_h\big)
\end{align*}
The approximation properties of the nodal interpolation operator $\II_h$, show

\begin{align*}
\int_0^T &\big((\alpha \vvv_{hk}^- + \mmm_{hk}^- \times\vvv_{hk}^-),(\mmm_{hk}^- \times \vphi)\big) + k\theta\int_0^T \big(\nabla \vvv_{hk}^-, \nabla(\mmm_{hk}^- \times \vphi)\big)\\
&\quad+ C_e\int_0^T\big(\nabla \mmm_{hk}^-, \nabla(\mmm_{hk}^- \times \vphi)\big) -\int_0^T \big(\HHH_{hk}^-, (\mmm_{hk}^-\times \vphi)\big)- \int_0^T \big(\pi(\mmm_{hk}^-), (\mmm_{hk}^-\times \vphi)\big)\\
&=\mathcal{O}(h)
\end{align*}
Passing to the limit and using the strong $\L^2(\omega_T)$-convergence of $(\mmm_{hk}^-\times \vphi)$ towards $(\mmm \times \vphi)$, in combination with Lemma~\ref{lem:subsequences} and the weak convergence property~\eqref{eq:assum3} of $\pi(\mmm_{hk}^-)$, this yields
\begin{align*}
 \int_0^T \big((\alpha \mmm_t + \mmm \times\mmm_t),(\mmm \times \vphi)\big) & = -C_e\int_0^T\big(\nabla \mmm, \nabla(\mmm \times \vphi)\big)\\
 &+ \int_0^T \big(\HHH, (\mmm\times \vphi)\big) + \int_0^T \big(\pi(\mmm), (\mmm\times \vphi)\big)
\end{align*}

Exploiting basic properties of the cross product, we conclude~\eqref{eq:weak_sol1}. The equality $\mmm(0, \cdot) = \mmm^0$ in the trace sense follows from the weak convergence $\mmm_{hk} \rightharpoonup \mmm$ in $\H^1(\omega_T)$ 
Analogously, we get $\HHH(0,\cdot) = \HHH^0$ in the trace sense. 
For the Eddy-current part,~\eqref{eq:alg:2} implies
\begin{align*}
& \mu_0 \int_0^T \big((\HHH_{hk})_t, \zzeta_h\big) + \sigma^{-1} \int_0^T(\nabla \times \HHH_{hk}^+, \nabla \times \zzeta_h) = -\mu_0 \int_0^T (\vvv_{hk}^-, \zzeta_h).
\end{align*}
The convergence properties from Lemma~\ref{lem:subsequences} in combination with the properties of the interpolation operator $\II_{\XX_h}$ from~\eqref{eq:interp} now reveal
\begin{align*}
\int_0^T \big((\HHH_{hk})_t, \zzeta_h\big) & \longrightarrow \int_0^T (\HHH_t, \zzeta),\\
\int_0^T (\nabla \times \HHH_{hk}^+, \nabla \times \zzeta_h) &\longrightarrow \int_0^T (\nabla \times \HHH, \nabla \times \zeta), \quad \text{ and }\\
\int_0^T (\vvv_{hk}^-, \zzeta_h) &\longrightarrow (\mmm_t, \zzeta),
\end{align*}
whence~\eqref{eq:weak_sol2}.

It remains to show the energy estimate~\eqref{eq:energy} which follows from the discrete energy estimate~\eqref{eq:discrete_energy} together with weak lower semi-continuity, cf.\ e.g.~\cite[Proof of Thm.\ 6]{maxwell} for details. This yields the desired result.
\end{proof}

\begin{remark}
Finally, we would like to comment on the choice of $\theta$.
\begin{enumerate}
\item For $0 \le \theta < 1/2$ one has to bound the negative term $(\theta - \frac12)k^2 \sum_{i=0}^{j-1} \norm{\nabla \vvv_h^i}{\L^2(\Omega)}^2$ on the left-hand side of~\eqref{eq:discrete_energy} in Lemma~\ref{lem:stability} in order to prove boundedness of the discrete quantities. This can be achieved by using an inverse estimate $\norm{\nabla \vvv_h^i}{\L^2(\Omega)}^2 \lesssim \frac{1}{h^2} \norm{\vvv_h^i}{\L^2(\Omega)}^2$. The upper bound can then be absorbed into the term $k \sum_{i = 0}^{j-1}\norm{\vvv_h^i}{\L^2(\Omega)}^2$ which yields convergence , cf.~\cite[Proof of Thm.\ 4.5]{lt} provided $k/h^2 \to 0$.
\item For the limiting case $\theta = \frac12$, Lemma~\ref{lem:stability} provides no boundedness of $\sqrt{k}\norm{\nabla \vvv_{hk}^-}{\L^2(\omega_T)}$. Therefore, the convergence 
\begin{align*}
k \theta \int_0^T \big(\nabla \vvv_{hk}^-, \nabla(\mmm_{hk}^- \times \vphi)\big) \to 0
\end{align*}
cannot be guaranteed. As suggested in~\cite{alouges2008}, this can be circumvented by an inverse estimate provided the fraction $\frac{k}{h}$ tends to zero.
\end{enumerate}
\end{remark}

%% file: 6_numerics.tex
\section{Numerical examples}\label{sec:numerics}
In order to carry out physically relevant experiments,
we choose $\mmm_0$ and $\HHH_0$
satisfying~\eqref{eq:consistency}. 
This can be achieved by taking
\begin{equation*}
\HHH_0 =\HHH_0^* -  \chi_{\omega}\mmm_0,
\end{equation*}
where $\diver \HHH_0^* = 0$ in $\Omega$. 
In our experiment, for simplicity, we choose
$\HHH_0^*$ to be a constant. 
We solve the standard problem $\#1$ proposed by the 
Micromagnetic Modeling Activity Group at the National Institute 
of Standards and Technology~\cite{mumag}. In this model, 
the initial conditions $\mmm_0$ and $\HHH_0$, and
the effective field $\HHH_{\text{eff}}$ are given as
\[
\mmm_0 = (1,0,0)\ \text{in}\ \omega, \quad
\HHH_0^* = (0,0,0)\ \text{in}\ \Omega,
\]
and
\[
\HHH_{\text{eff}}=\frac{2A}{\mu_0^{*}M_s^2}\Delta\mmm
+\HHH
+C_a\langle\mmm,\ppp\rangle\ppp
+\HHH_{ext}
\quad\text{with}\quad
\ppp = (1,0,0).
\]
The parameters for this problem are given below:
\begin{gather*}
\alpha=0.5,\quad
\sigma = 1,\quad
\mu_0 = 1.25667\times 10^{-6},\\
C_e = 5\times 10^2,\quad
A= 1.3\times 10^{-11},\quad
M_s = 8\times 10^{5}.
\end{gather*}

The domains $\omega$ and $\Omega$ are chosen 
(in $\mu m$) to be
\[
\omega = (0,2)\times (0,1)\times(0,0.02)
\]
and
\[
\Omega
=
(-0.2,2.2)\times(-0.2,1.2)\times(-0.04,0.06).
\]

The domain $\omega$ is uniformly partitioned into cubes of 
dimensions $0.1\times0.1\times0.02$, where each cube 
consists of six tetrahedra. We generate a nonuniform mesh for 
the magnetic domain~$\Omega$ in such a way that it is 
identical to the mesh for $\omega$ in the region $\omega$, 
and the mesh-size gradually increases away from
$\omega$.

For time discretization, we perform a uniform partition
of $[0,1]$ with timestep $k=0.01$. In each integration
step of Algorithm~\ref{alg}, we solved two linear
systems, one of size $2V\times2V$ where $V=462$ is the number
of vertices in the domain $\omega$, and another of size
$E\times E$ where $E=3991$ is the number edges in the
domain $\Omega$; see Figure~\ref{fig:mesh}.

Figure~\ref{fig:energy} depicts the evolution of the
exchange energy $\norm{\nabla\mmm_{h,k}(t)}{\omega}$, 
magnetic field energy
$\norm{\HHH_{h,k}(t)}{\Omega}$, and total energy
$\norm{\nabla\mmm_{h,k}(t)}{\omega}
+ \norm{\HHH_{h,k}(t)}{\Omega}
+ \norm{\nabla\times\HHH_{h,k}(t)}{\Omega}$.
The latter figure supports our theoretical result that
these energies are bounded.
\begin{figure}[p]
\centering
\includegraphics[width=4.5cm,height=4.5cm]{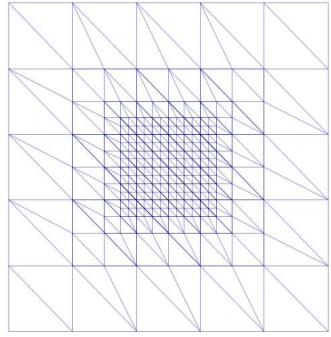}
\caption{Mesh for the domain $\Omega$ at
$z=0$.}\label{fig:mesh}
\end{figure}
\begin{figure}
\centering
\includegraphics[width=13cm,height=15cm]{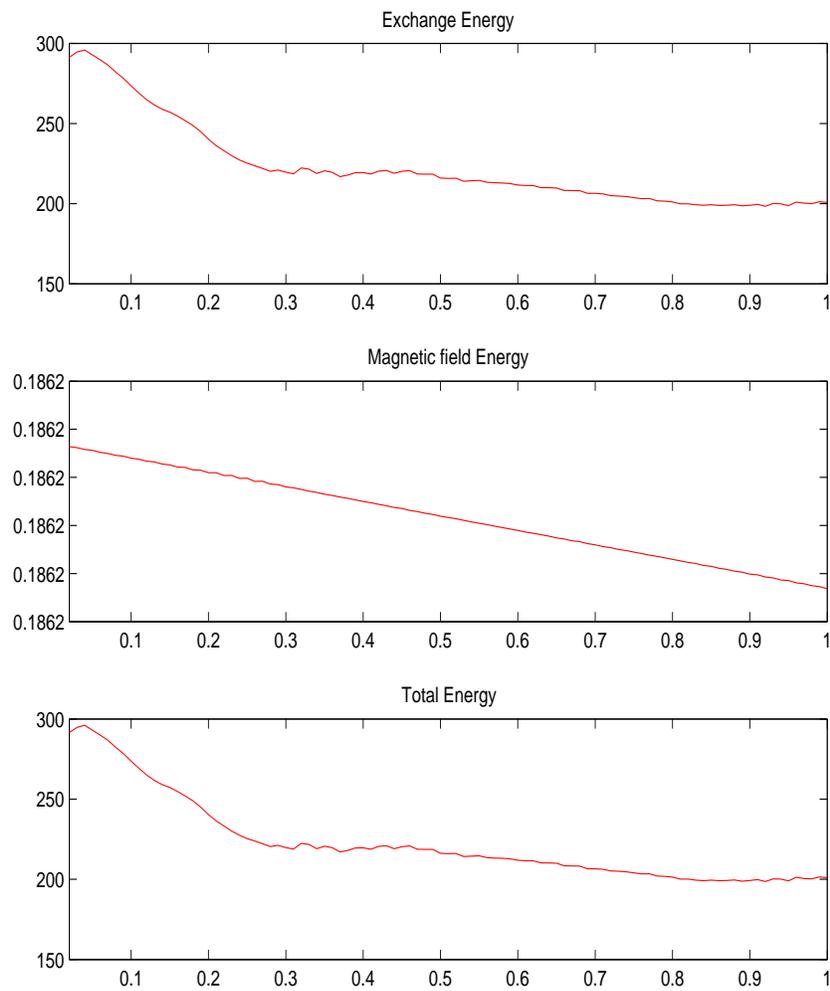}
\caption{Evolution of exchange, magnetic field, and 
total energies}\label{fig:energy}
\end{figure}

%% file: literature.tex
\newcommand{\bibentry}[2][!]{\ifthenelse{\equal{#1}{!}}{\bibitem{#2}}{\bibitem[#1]{#2}}}

%% file: eddyCurrent.bbl
\begin{thebibliography}{CKNS}

\bibentry{alouges2008}
\textsc{F.~Alouges}:
\emph{A new finite element scheme for Landau-Lifshitz equations}
Discrete and Continuous Dyn. Systems Series S, Vol. $\mathbf{1}$, pp. 187--196, (2008).

\bibentry{alouges2011}
\textsc{F.~Alouges, E.~Kritsikis, J.~Toussaint}:
\emph{A convergent finite element approximation for Landau-Lifshitz-Gilbert equation},
Physica B, $\mathbf{407}$, pp. 1345--1349, (2012).


\bibitem{as}
\textsc{F.~Alouges, A.~Soyeur}:
\emph{On global weak solutions for Landau-Lifshitz equations: existence and nonuniqueness},
Nonlinear Anal.\ \textbf{18}, pp. 1071--1084, (1992).

\bibentry{banas_mumag}
\textsc{L'.~Ba\v{n}as}:
\emph{An efficient multigrid preconditioner for Maxwell's equations in micromagnetism},
Mathematics and Computers in Simulation Vol. $\mathbf{80}$, pp. 1657--1663, (2010).

\bibentry{banas}
\textsc{L'.~Ba\v{n}as, S.~Bartels, A.~Prohl}:
\emph{A convergent implicit finite element discretization of the Maxwell-Landau-Lifshitz-Gilbert equation},
SIAM J. Numer. Anal. Vol. $\mathbf{46}$, pp. 1399--1422, (2008).


\bibentry{sllg}
\textsc{L'.~Ba\v{n}as, Z.~Brze\'{z}niak, A.~Prohl}:
\emph{Computational studies for the stochastic {L}andau-{L}ifshitz-{G}ilbert equation},
SIAM J. Sci. Comput., accepted, (2012).

\bibentry{maxwell}
\textsc{L'.~Ba\v{n}as, M.~Page, D.~Praetorius}:
\emph{A convergent linear finite element scheme for the Maxwell-Landau-Lifshitz-Gilbert equation},
extended preprint available on arXiv:1303.4009 (2013).

\bibentry{bartels}
\textsc{S.~Bartels}:
\emph{Stability and convergence of finite-element approximation schemes for harmonic maps},
SIAM J. Numer. Anal. Vol. $\mathbf{43}$, pp. 220--238, (2005).

\bibitem{bjp}
\textsc{S.~Bartels, J.~Ko, A.~Prohl}:
\emph{Numerical analysis of an explicit approximation scheme for the 
Landau-Lifshitz-Gilbert equation},
Math. Comp.\ \textbf{77}, 773--788, (2008).

\bibitem{bp}
\textsc{S.~Bartels, A.~Prohl}:
\emph{Convergence of an implicit finite element method for the 
Landau-Lifshitz-Gilbert equation},
SIAM J. Numer. Anal.\ \textbf{44}, 1405--1419, (2006).

\bibentry{brennerscott}
\textsc{S.~Brenner, L.~Scott}:
\emph{The mathematical theory of finite element methods},
Springer, New York, $^2(2002)$.

\bibentry{multiscale}
\textsc{F.~Bruckner, D.~Suess, M.~Feischl, T,~F\"uhrer, P.~Goldenits, M.~Page, D.~Praetorius}:
\emph{Multiscale modeling in micromagnetics: Well-posedness and numerical integration}, submitted for publication, preprint available at
arXiv:1209.5548, (2012).


\bibentry{carbou}
\textsc{G.~Carbou, P.~Fabrie}:
\emph{Time average in micromagnetism}, J. Differential Equations, $\mathbf{147}$, pp. 383--409, (1998).

\bibitem{cimrak}
\textsc{I.~Cimrak}:
\emph{A survey on the numerics and computations for the Landau-Lifshitz 
equation of micromagnetism}, Arch. Comput. Methods Eng.\ \textbf{15}, pp. 277--309, (2008).

\bibitem{DSM1}
\textsc{M.~d'Aquino, C.~Serpico, G.~Miano}:
\emph{Geometrical integration of {L}andau-{L}ifshitz-{G}ilbert equation
  based on the mid-point rule},
{J. Comput. Phys.}, \textbf{209}, pp. 730--753, (2005).


\bibitem{gc}
\textsc{C.J.~Garc\'ia-Cervera}
\emph{Numerical micromagnetics: a review},
Bol. Soc. Esp. Mat. Apl. SeMA \textbf{39}, pp. 103--135, (2007).

\bibitem{Gil55}
T.~Gilbert.
\newblock A {L}agrangian formulation of the gyromagnetic
equation of the
  magnetic field.
\newblock {\em Phys Rev}, \textbf{100}:1243--1255, 1955.

\bibentry{petra}
\textsc{P.~Goldenits}:
\emph{Konvergente numerische Integration der Landau-Lifshitz-Gilbert Gleichung} (in German),
PhD thesis, Institute for Analysis and Scientific Computing, Vienna University of Technology, (2012).

\bibitem{mathmod2012}
\textsc{P.~Goldenits, G.~Hrkac, M.~Mayr, D.~Praetorius, D.~Suess}:
\emph{An effective integrator for the Landau-Lifshitz-Gilbert equation},
Proceedings of Mathmod 2012 Conference, (2012).

\bibitem{gamm2011}
\textsc{P.~Goldenits, D.~Praetorius, D.~Suess}:
\emph{Convergent geometric integrator for the Landau-Lifshitz-Gilbert equation
in micromagnetics}, Proc.\ Appl.\ Math.\ Mech.\ \textbf{11}, 775--776, (2011).

\bibitem{hubertschaefer}
\textsc{A.~Hubert, R.~Sch\"afer}:
\textit{Magnetic Domains. The Analysis of Magnetic Microstructures},
Corr. 3rd printing, Springer, Heidelberg, (1998).

\bibitem{mp06}
\textsc{M.~Kruzik, A.~Prohl}:
\emph{Recent developments in the modeling, analysis, and numerics of 
ferromagnetism},
SIAM Rev. \textbf{48}, 439--483, (2006).

\bibitem{LL35}
L.~Landau and E.~Lifschitz.
\newblock On the theory of the dispersion of magnetic
permeability in
  ferromagnetic bodies.
\newblock {\em Phys Z Sowjetunion}, \textbf{8}:153--168, 1935.

\bibitem{lt}
\textsc{Kim-Ngan~Le, T.~Tran}:
\emph{A convergent finite element approximation for the quasi-static Maxwell-Landau-Lifshitz-Gilbert equations}, submitted for publication, preprint available at
arXiv:1212.3369v1, (2012).

\bibentry{monk}
\textsc{P.~B.~Monk}:
\emph{Finite Element Methods for Maxwell's Equations},
Oxford University Press, Oxford, UK, (2003).

\bibentry{mumag}
\textsc{$\mu$Mag-Website of NIST-Institute}: {http://www.ctcms.nist.gov/{$_{\widetilde{~}}$}rdm/mumag.org.html}

\bibitem{diss}
\textsc{M.~Page}:
\textit{On dynamical micromagnetism},
PhD thesis, Institute for Analysis and Scientific Computing, Vienna University of Technology, (2013). 

\bibitem{prohl}
\textsc{A.~Prohl}:
\textit{Computational micromagnetism},
Advances in Numerical Mathematics. 
B.\ G.\ Teubner, Stuttgart, (2001).


%

\bibitem{visintin}
\textsc{A.~Visintin}:
\emph{On Landau-Lifshitz equations for ferromagnetism},
Japan J.\ Appl.\ Math., Vol.\ \textbf{2}, 69--84, (1985).


\end{thebibliography}
